\documentclass{amsart}
\usepackage{amssymb}
\usepackage{amsmath}
\usepackage{amsthm}
\usepackage[utf8]{inputenc}
\usepackage[T1]{fontenc}

\newtheorem{theorem}{Theorem}[section]
\newtheorem{lemma}[theorem]{Lemma}

\theoremstyle{definition}
\newtheorem{definition}[theorem]{Definition}

\newtheorem{proposition}[theorem]{Proposition}
\newtheorem{corollary}[theorem]{Corollary}
\theoremstyle{remark}
\newtheorem{remark}[theorem]{Remark}

\numberwithin{equation}{section}

\begin{document}

\title[Presentation of the Iwasawa algebra $\Lambda(G(1))$]{Presentation of the Iwasawa algebra of the first congruence kernel of a semi-simple, simply connected Chevalley group over $\mathbb{Z}_p$.}

\author{Jishnu Ray}
\address{Département de Mathématiques, Université Paris-Sud $11$, Bât. $430$, $91405$ Orsay CEDEX France
}
\email{jishnu.ray@u-psud.fr, jishnuray1992@gmail.com}
 \thanks{The author receives funding from the Université Paris-Sud by the EDMH PhD fellowship of the Université Paris-Saclay.}



\dedicatory{This paper is dedicated to my father, who loves mathematics with all his heart.}


\begin{abstract}
It is a general principle that objects coming from semi-simple, simply connected (split)  groups have explicit presentations like Serre's presentation of semi-simple algebras and Steinberg's presentation of Chevalley groups. In this paper we give an explicit presentation (by generators and relations) of the Iwasawa algebra for the first congruence kernel of a semi-simple, simply connected Chevalley group over $\mathbb{Z}_p$, extending the proof given by Clozel for the group $\Gamma_1(SL_2(\mathbb{Z}_p))$, the first congruence kernel of $SL_2(\mathbb{Z}_p)$ for primes $p>2$ . 
\end{abstract}

\maketitle
\section{Introduction}
Lazard, in his seminal work \cite{Lazard}, II.$2.2$ studied non-commutative Iwasawa algebras for pro-$p$  groups. They are completed group algebras
\begin{equation*}
\Lambda(P)=\varprojlim \mathbb{Z}_p[P/N]
\end{equation*}
where $\mathbb{Z}_p$ is the ring of $p$-adic integers, $P$ is a pro-$p$ group and the inverse limit is taken over the open normal subgroup $N$ of $P$. It seems that explicit description, by generators and relations, of these algebras were inaccessible. However, Serre's presentation of semi-simple algebras and Steinberg's presentation of Chevalley groups \cite{Ser}, \cite{Yale} make us believe that the objects coming from semi-simple split groups have explicit presentation. Indeed, for any odd prime $p$, Clozel in his paper \cite{Laurent} gives an explicit presentation of the Iwasawa algebra of the subgroup of level $1$ of $SL_2(\mathbb{Z}_p)$, which is $\Gamma_1(SL_2(\mathbb{Z}_p))=\ker(SL_2(\mathbb{Z}_p) \rightarrow SL_2(\mathbb{F}_p))$. Notice that as $\Gamma_1(SL_2(\mathbb{Z}_p))\cong \mathbb{Z}_p^3$, its (non-commutative) Iwasawa algebra is isomorphic (as topological $\mathbb{Z}_p$-\textit{module}) to the ring of power series in three variables. The key ingredient in Clozel's proof was to compute the relations between those variables viewing them as elements of the Iwasawa algebra. The case $p=2$ has to be omitted because then  $\Gamma_1(SL_2(\mathbb{Z}_p))$ will have $p$-torsion and thus its Iwasawa algebra is not an integral domain which prevents one from using deep results of Lazard.

Our main result is to give an explicit presentation, by generators and relations, of the Iwasawa algebra for the subgroup of level $1$ of a semi-simple, simply connected, split Chevalley group over $\mathbb{Z}_p$. Lazard, in his paper \cite{Lazard} defines, for a compact locally $\mathbb{Q}_p$-analytic group $H$, a function $\omega:  H -\{1\} \rightarrow (\frac{1}{p-1},\infty) \subset \mathbb{R}$ satisfying certain properties (cf. \cite{Lazard} III.$2.1.2$) and calls it a $p$-valuation. Let $d$ be the dimension of $H$ (as a locally analytic manifold). Lazard also defines an ordered basis of $H$ with respect to the $p$-valuation $\omega$. Those are an ordered sequence of elements $h_1,...,h_d \in H-\{1\}$ such that the following conditions hold: 
\begin{enumerate}
\item $\psi:\mathbb{Z}_p \xrightarrow{\sim} H, (x_1,...,x_d) \longmapsto h_1^{x_1}\cdots h_d^{x_d}$,\\
\item $\omega(h_1^{x_1}\cdots h_d^{x_d})=\min_{1 \leq i \leq d}(\omega(h_i)+val_p(x_i))$,
\end{enumerate}
where the map $\psi$ is a homeomorphism. Let $G$ be a semi-simple, simply connected Chevalley group over $\mathbb{Z}_p$ with a split maximal torus  $T$, $\Phi$ be the root system  of the Lie algebra of $G$ with respect to the Lie algebra of $T$, $\Pi$ be a set of simple roots of $\Phi$. Under a faithful representation of group schemes $G \hookrightarrow GL_n$ over $\mathbb{Z}$, consider $G(1)$-the pullback of the congruence kernel at level $1$ of $GL_n(\mathbb{Z}_p)$. The natural filtration of $G(1)$ by deeper congruence subgroups enables us to define a $p$-valuation $\omega$ on $G(1)$. 

Let us now present a brief overview of the results of our paper.
\begin{itemize}
\item The first result in our paper is to find an ordered basis of $G(1)$ with respect to its $p$-valuation. Such an ordered basis is constructed in theorem \ref{eq:orderedbasisthm}. For the proof we have used the triangular decomposition of $G(1)$ and the assumption that our group $G$ is simply connected. 

\item The second result concerns the presentation of the Iwasawa algebra of $G(1)$ in terms of generators and relations.  Let $\mathcal{A}$ be the universal non-commutative $p$-adic algebra over $\mathbb{Z}_p$ in the variables $\{V_{\alpha},W_{\delta};\alpha \in \Phi,\delta \in \Pi\}$; the ordering being given by the increasing height function on the roots. The algebra $\mathcal{A}$ has a topology given by the filtration by the powers of its unique maximal ideal. The algebra $\mathcal{A}$ is naturally sent to $\Lambda(G(1))$ (cf. section $6$). Computing in $\Lambda(G(1))$, we obtain the relations between the variables in $\mathcal{A}$ precisely given by $(\ref{eq:firstrelation},\ref{eq:secondequation},\ref{eq:thirdequation},\ref{eq:forthrelation})$ and let $\mathcal{R}$ be the closed two-sided ideal generated by the relations in $\mathcal{A}$. Then our main result is (cf. Theorem $\ref{eq:maintheorem}$)

\textbf{\textit{Theorem}}. For $p>2$, the Iwasawa algebra $\Lambda(G(1))$ is naturally isomorphic to $\mathcal{A}/\mathcal{R}$.

To compute the relations between the variables in $\mathcal{A}$, we have used the Steinberg's Chevalley relations for simply connected groups (\cite{Yale}). The advantage of this description of the Iwasawa algebra is not only its simplicity but also the fact that it allows one to do explicit computations about its center. For example, for the group $\Gamma_1(SL_2(\mathbb{Z}_p))$, Clozel used his presentation to show that the center of the Iwasawa algebra of $\Gamma_1(SL_2(\mathbb{Z}_p))$ is composed of the multiple of the Dirac measure at $1$, thus giving a different proof of Ardakov's result (cf. \cite{Ardakov}).
\end{itemize}
Hence the objective of this paper is to solve two problems: firstly finding an ordered basis for the $p$-valuable group $G(1)$, secondly to use it to give an explicit presentation of the Iwasawa algebra of $G(1)$.
 
\textbf{The structure of this paper}. In Section $2$ we go for a quick tour through the notions of $p$-valuable groups and Chevalley groups that we need. In section $3$ we construct an ordered basis of $G(1)$ in theorem \ref{eq:orderedbasisthm}, an alternative proof of which is provided in section $4$ using the theory of uniform pro-$p$ groups (cf. \cite{Analy}). In section $5$ we use Steinberg's relations of Chevalley groups to compute the relations in $\Lambda(G(1))$.  Finally, in section $6$, we first provide an explicit presentation of the Iwasawa algebra of $G(1)$ with coefficients in $\mathbb{F}_p$ (cf. theorem \ref{eq:secondimptheorem}) and then we lift its coefficients to $\mathbb{Z}_p$ to prove our main result, which is theorem \ref{eq:maintheorem}.
\section{Notations and preliminary definitions}
Let $G$ be any abstract group. We recall that an application 
\begin{equation*}
\omega:G \rightarrow \mathbb{R}_+^{*} \cup \{+\infty \}
\end{equation*}
is called a filtration (cf. \cite{Lazard}, II.$1.1.1$) if the following conditions hold for all $x,y \in G$:

\begin{enumerate}
\item $\omega(xy^{-1})\geqslant \min (\omega(x),\omega(y))$,
\item $ \omega(x^{-1}y^{-1}xy)\geqslant \omega(x)+\omega(y)$.
\end{enumerate}

If $\omega$ is such a filtration, then  $G$ is called a filtered group and $\omega(x)$ is called the filtration of the element  $x$. 

For any $v \in \mathbb{R}_+^{*}$, we define the subgroups
\begin{equation*}
G_v=\{x \in G|\omega(x)\geqslant v\} \ \text{and} \ G_{v+}=\{x \in G|\omega(x)> v\}.
\end{equation*}
Then the subgroups $G_v$ satisfy the following three relations:
\begin{enumerate}
\item $G=\cup_{v>0}G_v$,
\item $[G_v,G_{v^{'}}]\subset G_{v+v^{'}}$ for $v,v^{'} \in \mathbb{R}_+^{*}$,
\item $G_v=\cap_{v^{'} \leq v}G_{v^{'}}$ for $v \in \mathbb{R}_+^{*}$,
\end{enumerate}
 where $[G_v,G_{v^{'}}]$ is the commutator subgroup of $G_v$ and $G_{v^{'}}$. Conversely, if there exists a family of subgroups $G_v$ satisfying the above three relations 
then we can define a filtration $\omega$ by the formula 
\begin{center}
$\omega(x)=\sup v$ for $x \in G_v$.
\end{center}
(Cf. Lazard, \cite{Lazard}, II.$1.1.2.4$). Such a filtration $\omega$ is called separated if $\omega(x)=+\infty$ implies $x=1$. For a prime number $p$ and for all $v \in \mathbb{R}_+$, we define

 \begin{center}
 $\varphi(v)=\min(v+1,pv)$.
 \end{center} 
 The filtration $\omega$ is called a $p$-filtration if it verifies the following axiom:
\begin{center}
$\omega(x^p)\geqslant \varphi (\omega(x))$ for $ x \in G$,
\end{center}
that is 
\begin{center}
$\omega(x^p) \geq \min (\omega(x)+1,p\omega(x))$.
\end{center}
In that case, $G$ is said to be a $p$-filtered group with the $p$-filtration $\omega$.

\begin{definition}
  A filtered group $G$ is $p$-valued if, for all $x \in G$, $\omega$ satisfies the following three axioms:
\begin{align*}
\omega(x)&<\infty,  \ x \neq 1,\\
\omega(x)&>(p-1)^{-1},\\
\omega(x^p)&=\omega(x)+1.
\end{align*}
\end{definition}
Such a group $G$ is also called $p$-valuable, and $\omega$ is called a $p$-valuation on $G$ with the convention $\omega(1)=\infty$.  Henceforth we assume that $G$ is a profinite group and $\omega$ is a $p$-valuation on $G$ which defines the topology of $G$. 

The $G_v$'s form a decreasing filtration of $G$, so there exists the unique topology on $G$ (the topology defined by the filtration) such that $G_v$'s form a fundamental system of open neighborhoods of identity. It is the topology defined by $\omega$.
 We put for each $v>0$,
 \begin{center} 
 	gr$_v G :={G_v}/{G_{v+}} $ and gr$(G):=\oplus_{v>0}$ gr$_{v}G$.
 \end{center}
The commutator induces a Lie bracket on gr$(G)$  which gives gr$(G)$ the structure of a  Lie algebra over $\mathbb{F}_p$. The map $P$ defined by  
\begin{equation*}
P:\text{gr}(G) \rightarrow \text{gr}(G), \ P(gG_{v+})=g^pG_{(v+1)+}
\end{equation*}
is an $\mathbb{F}_p$-linear map on gr$(G)$, which gives gr$(G)$ the structure of a graded Lie algebra over $\mathbb{F}_p[P]$ (cf. \cite{Lazard}, III.$2.1.1$). 

\begin{definition} 
 The pair $(G, \omega)$ is called of finite rank if gr$(G)$ is finitely generated over $\mathbb{F}_p[P]$.
 \end{definition}
 
 As gr$(G)$ is torsion free module over $\mathbb{F}_p[P]$ (cf.  \cite{Lie}, remark $25.2$), being finitely generated torsion free module over a principal ideal domain $\mathbb{F}_p[P]$, it  is free over $\mathbb{F}_p[P]$. We  define 
 \begin{center}
 rank($G, \omega ):=$rank$_{\mathbb{F}_p[P]}$gr$G$.
 \end{center}
 For $g_1,...,g_d \in G$,  we consider the continuous map 
 \begin{align}\label{eq:basis}
 \mathbb{Z}_p^d \rightarrow G, \ (x_1,... ,x_d)\mapsto g_1^{x_1} \cdots g_d^{x_d}.
 \end{align}
 The  map  above depends on the order of $g_1,...,g_d$ and hence it is not a group homomorphism. 
 
 \begin{definition} 
 \label{eq:orderedbasis}
 The sequence of elements $(g_1,... ,g_d)$ in $G$ is called an ordered basis of $(G, \omega )$ if the map in (\ref{eq:basis}) is a bijection (and hence, by compactness, a homeomorphism) and
  \begin{equation}\label{eq:equationlazardbasis}
 \omega (g_1^{x_1},\cdots, g_d^{x_d})=\min_{1\leqslant i\leqslant d}(\omega (g_i)+val_p(x_i))\ \text{for} \  x_i\in \mathbb{Z}_p.
 \end{equation}
 \end{definition}
  If $(G, \omega )$ is of finite rank, then the rank of gr$G$ over $\mathbb{F}_p[P]$ is finite. Following \cite{Lie} (proposition $26.5$), we can relate the basis of $G$ to the basis of gr$(G)$ over $\mathbb{F}_p[P]$. In fact, the following lemma holds.
 \begin{lemma}\label{eq:jishnu}
  The sequence $(g_1,...,g_d)$ is an ordered basis of $(G, \omega )$ if and only if $\sigma (g_1),...,\sigma (g_d)$ is a basis of the $\mathbb{F}_p[P]$-module gr$(G)$ where for $g \neq 1$,  $\sigma (g):=gG_{\omega (g)+} \in $ gr$(G)$.
 \end{lemma}
In particular a $p$-valuation $\omega$ of $G$ is a $p$-filtration, separated for which gr$_v(G)=0$ for $ v\leqslant (p-1)^{-1}$, and gr$(G)$ is torsion free (cf. the discussion after $2.1.2.3$, chap III of \cite{Lazard}). 

Let $G$ be  a $p$-valued group, complete, with discrete filtration. Let $g_1,... ,g_d$ be an ordered basis. Then $G$ is defined to be $p$-saturated (cf.  \cite{Lazard}, III.2.2.7) if the valuations $\omega(g_i)$ satisfy the following relation:
\begin{equation}
\label{eq:saturated}
(p-1)^{-1}<\omega(g_i)\leqslant p(p-1)^{-1},\ i \in [1,d].
\end{equation}

In the remaining part of this section, we introduce the basic notion of Chevalley groups over $\mathbb{Z}$ that we need throughout the paper (cf. \cite{Chevalley},\cite{Kos}) following section $1$ of \cite{Abe} . Steinberg gives the construction in \cite{Yale}. 

Let $G_{\mathbb{C}}$ be a connected complex semi-simple Lie group, $T_{\mathbb{C}}$ a maximal torus of $G_{\mathbb{C}}$. Let $\mathfrak{g}_{\mathbb{C}},\mathfrak{t}_{\mathbb{C}}$ be the Lie algebras of $G_{\mathbb{C}}$ and $T_{\mathbb{C}}$ respectively. Let $\Phi$ be the  root system of $\mathfrak{g}_{\mathbb{C}}$ with respect to $\mathfrak{t}_{\mathfrak{C}}$, $\Pi=\{\delta_1,...,\delta_l\}$ be a set of simple roots of $\Phi$, $\mathfrak{g}_{\mathbb{Z}}$ be a Chevalley lattice (cf. Theorem $1$ of \cite{Yale}) of $\mathfrak{g}_{\mathbb{C}}$ generated by the Chevalley basis (cf. p.6 of \cite{Yale})   $\{H_{\delta_1},...,H_{\delta_l},X_{\gamma},\gamma  \in \Phi\}$. For each $\gamma \in \Phi$, the element $H_{\gamma}=[X_{\gamma},X_{-\gamma}]$ is contained in the submodule $\mathfrak{t}_{\mathbb{Z}}=\mathfrak{g}_{\mathbb{Z}}\cap \mathfrak{t}_{\mathbb{C}}.$ We then have, by definition of the Chevalley basis,
\begin{flushleft}
$(a)$ $ \gamma(H_{\gamma})=2, \gamma \in \Phi,$\\
$(b)$ if $\gamma_1,\gamma_2$ are roots, then $\gamma_2(H_{\gamma_1})=v-u$, where $v,u$ are the non-negative integers such that $\gamma_2+i\gamma_1$ is a root if and only if $-v \leq i \leq u$, or\\
$(c)$ if $\gamma_1,\gamma_2$ and $\gamma_1+\gamma_2$ are roots, $[X_{\gamma_1},X_{\gamma_2}]=N_{\gamma_1,\gamma_2}X_{\gamma_1+\gamma_2},$ where $N_{\gamma_1\gamma_2}=\pm(v+1)$. 
\end{flushleft}
Let $\rho$ be a faithful representation of $G_{\mathbb{C}}$ in an $n$-dimensional vector space $V$ over $\mathbb{C}$, $d\rho$ the differential of $\rho$ which is a representation of $\mathfrak{g}_{\mathbb{C}}$ in $V$. Then, there exists (cf. p.17 of \cite{Yale}) a free $n$-dimensional, $\mathbb{Z}$-module $V_{\mathbb{Z}}$ generated by $\{v_1,...,v_n\}$ in $V$ which is stable under the action of the universal enveloping algebra $\mathfrak{U}_{\mathbb{Z}}$. We also have $$d\rho(H_{\gamma})v_i=\lambda_i(H_{\gamma});1 \leq i \leq n,\gamma \in \Phi$$ where $\lambda_i$ are linear functions on $\mathfrak{t}_{\mathbb{C}}$ with $\lambda_i(H_{\gamma}) \in \mathbb{Z}$. The base $\{v_1,...,v_n\}$ of $V_{\mathbb{Z}}$ determines the coordinates $x_{ij}(1 \leq i,j \leq n)$ on $GL(V)$ and the restrictions of $x_{ij}$ to $G_{\mathbb{C}}$ generate a subring $\mathbb{Z}[G]$ of the affine algebra $\mathbb{C}[G]$ of $G_{\mathbb{C}}$. The ring $\mathbb{Z}[G]$ has a structure of a Hopf algebra (cf. section $1$ of \cite{Abe}) and defines a group scheme $G$ over $\mathbb{Z}$. Namely,
\begin{center}
$R \rightarrow G(R)=Hom(\mathbb{Z}[G],R)$
\end{center}
is a contravariant functor from the category of commutative unitary rings into the category of groups. We call the group $G$ the Chevalley group scheme associated with $G_{\mathbb{C}}$.

Now, for any $t \in \mathbb{C}$, $x_{\gamma}(t)=exp$ $td\rho(X_{\gamma}),(\gamma \in \Phi),$ is an element of $G_{\mathbb{C}}$ and the coordinates of $x_{\gamma}(t)$ are polynomial functions on $t$ with coefficients in $\mathbb{Z}$. Let $\mathbb{Z}[\zeta]$ be the algebra over $\mathbb{Z}$ generated by one variable $\zeta$. Then there exists a homomorphism of $\mathbb{Z}[G]$ onto $\mathbb{Z}[\zeta]$ which assigns to each $x_{ij}$ the $(i,j)$-coordinate of $x_{\gamma}(\zeta)$. The homomorphism induces an injective homorphism of groups 
\begin{center}
$G_{\gamma}(R)=Hom(\mathbb{Z}[\zeta],R) \rightarrow G(R)=Hom(\mathbb{Z}[G],R)$
\end{center}
We denote also by $x_{\gamma}(t),t \in R$, the element of $G(R)$ corresponding to an element of $G_{\gamma}(R)$ such that $\zeta \rightarrow t$.

Let $P$(resp. $X,P_r$) the additive group generated by the weights of all the representations of $G$ (resp. the weights of $\rho$, the roots of $\mathfrak{g}_{\mathbb{C}})$. Then these are all free abelian groups of rank $l$ such that $P\supseteq X\supseteq P_r$. The group $G$ is said to be simply connected or universal if $P=X$. Henceforth we fix the ring $R$ to be $\mathbb{Z}_p$, $p$ being an odd prime, and we always work with the simply connected group denoted $G$, and $G(\mathbb{Z}_p)$ will denote its $\mathbb{Z}_p$-points.

For $\lambda \in \mathbb{Q}_p^*, \gamma \in \Phi$, we define 
\begin{equation*}
	h_{\gamma}(\lambda):=w_{\gamma}(\lambda)w_{\gamma}(1)^{-1} 
\end{equation*}

where
\begin{equation*}
w_{\gamma}(\lambda):=x_{\gamma}(\lambda)x_{-{\gamma}}(-\lambda^{-1})x_{\gamma}(\lambda).
\end{equation*}	
Given our embedding  $\rho:G \hookrightarrow GL_n$, let us define, for $k \in \mathbb{N}$,
\begin{center}
$\Gamma(k):=\ker(GL_n(\mathbb{Z}_p)\rightarrow GL_n(\mathbb{Z}_p/p^k\mathbb{Z}_p))$
\end{center}
(the $\mathbb{Z}$-structure on $GL_n$ being given by $V_{\mathbb{Z}}$) and $G(k):=G(\mathbb{Z}_p)\cap \Gamma(k)$. Then $G(k)$ is called the $k^{th}$ congruence kernel of $G(\mathbb{Z}_p)$ which satisfies a descending filtration 
$G(1) \supseteq G(2) \supseteq \cdots $.

\section{Ordered basis of $G(1)$}		

Let us define $\omega$, a function on the first congruence kernel $G(1)$, by
\begin{equation}\label{eq:defomega}
\omega(x)=k \ \text{for} \ x \in G(k) \backslash G(k+1).
\end{equation}
We then show in theorem \ref{lemma1.3}  that $\omega$ is a $p$-valuation on $G(1)$. Furthermore, we show in theorem \ref{eq:orderedbasisthm} that $\{x_{\beta}(p),h_{\delta}(1+p),x_{\alpha}(p); \ {\beta} \in \Phi^-,{\delta} \in \Pi, {\alpha} \in \Phi^+\}$  forms an ordered basis for $(G(1), \omega)$, where the order is given by the height function on the roots.

\begin{theorem}
\label{lemma1.3}
The valuation $\omega$ defined in ($\ref{eq:defomega}$) is a $p$-valuation on $G(1)$.
\end{theorem}

\begin{proof}
First we recall that the function (denoted again by $\omega$) on $\Gamma(1)$, defined as
\begin{center}
$\omega(x)=k$ for $x \in \Gamma(k)\backslash \Gamma(k+1)$
\end{center}
is a $p$-valuation on $\Gamma(1)$ (cf. p.$171$ of \cite{Lie}). Therefore, $\Gamma(k)$ satisfies the following three conditions:

\begin{center}
$\Gamma(1)=\cup_{k\geqslant 1}\Gamma(k)$,\\
$[\Gamma(k),\Gamma(k^{'})]\subset \Gamma_{k+k^{'}}$ for $k,k^{'} \in \mathbb{N}$,\\
$\Gamma(v)=\cap_{w \leq v}\Gamma(w)$ for $v,w \in \mathbb{N}$.
\end{center}
By definition of $G(k)$ we have 
\begin{center}
$G(1)=\cup_{k\geqslant 1}G(k),$\\
$G(v)=\cap_{w \leq v}G(w)$ for $v,w \in \mathbb{N}.$
\end{center}
Also,
\begin{align*}
[G(k),G(k^{'})]&\subseteq G(\mathbb{Z}_p)\cap [\Gamma(k),\Gamma(k^{'})],\\
&\subseteq G(\mathbb{Z}_p)\cap \Gamma(k+k^{'}),\\
&=G(k+k^{'}).
\end{align*}
This shows that (by section $2$) $\omega$ is a filtration on $G(1)$. Obviously, if $x\in G(1)$ then $\omega(x)=+\infty$ implies $x \in \cap_k G(k) \subset \cap_k\Gamma(k)=1$ which in particular shows that $\omega$ is separated. For $x \in G(1), x \neq 1$ we have $\omega(x)<\infty$. The valuation  $\omega(g)$ is strictly greater than $(p-1)^{-1}$ for all $g \in G(1)$. To prove that $\omega(g^p)=\omega(g)+1$, we  use the fact that  $\Gamma(1)$ is $p$-valuable. 

Suppose $\omega(x)=k$ i.e. $x \in G(k)\backslash G(k+1)$ where $G(k)=\Gamma(k)\cap G(\mathbb{Z}_p)$. Since $\Gamma(1)$ is $p$-valuable, $x^p\in \Gamma(k+1)\backslash \Gamma(k+2)$. This implies that $x^p \in G(k+1)\backslash G(k+2)$. Hence, we obtain $\omega(x^p)=\omega(x)+1$. This finishes the proof that $\omega$ is a $p$-valuation on $G(1)$.
\end{proof}

\begin{theorem}
\label{triangulard}
Any element $g \in G(k)$ can be uniquely written as 
\begin{center}
$g=\prod_{\beta \in \Phi^-}x_{\beta}(u_\beta)\prod_{\delta \in \Pi}h_\delta(1+v_{\delta})\prod_{\alpha \in \Phi^+}x_{\alpha}(w_{\alpha})$,
\end{center}
where $u_{\beta},w_{\alpha},v_{\delta} \in p^k\mathbb{Z}_p$. The order of the above product is given by the height function on the roots.
\end{theorem}
\begin{proof}
Let us define
\begin{align*}
U(\mathbb{Z}_p,p^k\mathbb{Z}_p)&:=\langle x_{\alpha}(t);t \in p^k\mathbb{Z}_p, \alpha \in \Phi^+\rangle,\\
V(\mathbb{Z}_p,p^k\mathbb{Z}_p)&:=\langle x_{\beta}(v);v \in p^k\mathbb{Z}_p, \beta \in \Phi^- \rangle,\\
T(\mathbb{Z}_p,p^k\mathbb{Z}_p)&:=\langle h_{\gamma}(u);u \in 1+p^k\mathbb{Z}_p, \gamma \in \Phi \rangle.
\end{align*}
From corollary $3.3$ and section $3.21$ of \cite{Abe} we have a unique triangular decomposition
\begin{align}\label{eq:triangular}
G(k)=V(\mathbb{Z}_p,p^k\mathbb{Z}_p)T(\mathbb{Z}_p,p^k\mathbb{Z}_p)U(\mathbb{Z}_p,p^k\mathbb{Z}_p).
\end{align}

Proposition $2.5$ (and the discussion following it) of \cite{Abe} gives that each element of $U(\mathbb{Z}_p,p^k\mathbb{Z}_p)$ can be written uniquely in the form $\prod_{\alpha \in \Phi^+}x_{\alpha}(t_i), t_i \in p^k\mathbb{Z}_p$, where the product is taken over all the positive roots with the order according to increasing height, the uniqueness  criterion follows from p.$26$-corollary $2$ of \cite{Yale}. Because of symmetry the similar statement will also hold for $V(\mathbb{Z}_p,p^k\mathbb{Z}_p)$. 

For $h_{\gamma}(u) \in T(\mathbb{Z}_p,p^k\mathbb{Z}_p),\gamma \in \Phi$, lemma $28$ and the proof of corollary $5$, p.$44$ of \cite{Yale} give that there exist integers $n_i$ such that 
\begin{center}
$h_{\gamma}(u)=\prod_{i=1}^lh_{\delta_i}(u)^{n_i}=\prod_{i=1}^lh_{\delta_i}(u^{n_i})$,
\end{center}
where $\delta_1,\cdots ,\delta_l$ are the simple roots, $l$ being the cardinality of the set of simple roots. The above expression is unique because our $G$ is simply connected (cf. Corollary of lemma $28$ p.$44$ of \cite{Yale}). Thus by $\ref{eq:triangular}$, each element $g \in G(k)$ can be uniquely written as 
\begin{center}
$g=\prod_{{\beta} \in \Phi^-}x_{\beta}(u_{\beta})\prod_{{\delta} \in \Pi}h_{\delta}(1+v_{\delta})\prod_{{\alpha} \in \Phi^+}x_{\alpha}(w_{\alpha})$,
\end{center}
where $u_{\beta},w_{\alpha},v_{\delta} \in p^k\mathbb{Z}_p$. The order of the above product is given by the height function on the roots.
\end{proof}
\begin{remark}
The paper \cite{Abe} by Abe gives actually the decomposition 
\begin{center}
$G(k)=U(\mathbb{Z}_p,p^k\mathbb{Z}_p)T(\mathbb{Z}_p,p^k\mathbb{Z}_p)V(\mathbb{Z}_p,p^k\mathbb{Z}_p)$.
\end{center}
Its proof uses proposition $1$ of \cite{Chevalley} and by just following the proof of corollary $3.3$ of \cite{Abe}, one can easily show that there is no harm if we interchange the places of $V(\mathbb{Z}_p,p^k\mathbb{Z}_p)$ and $U(\mathbb{Z}_p,p^k\mathbb{Z}_p)$ in the above decomposition.
\end{remark}

\begin{theorem}
\label{eq:th1.3.3.3}
Let $g \in G(k)$. Then using the decomposition given by theorem $\ref{triangulard}$,  if 
\begin{center}
$g=\prod_{{\beta} \in \Phi^-}x_{\beta}(u_{\beta})\prod_{{\delta} \in \Pi}h_{\delta}(1+v_{\delta})\prod_{{\alpha} \in \Phi^+}x_{\alpha}(w_{\alpha})$,
\end{center}
then 
\begin{center}
$\omega(g)=\min_{\{{\beta} \in \Phi^-,{\alpha} \in \Phi^+, {\delta} \in \Pi\}}\{val_p(u_{\beta}),val_p(v_{\delta}), val_p(w_{\alpha})\}$.
\end{center}
\end{theorem}

\begin{proof}
Let $g \in G(k)\backslash G(k+1)$ so that $\omega(g)=k$. Then any one of the elements of 
\begin{center}
$\{u_{\beta},v_{\delta},w_{\alpha};{\beta} \in \Phi^-,{\delta} \in \Pi, {\alpha} \in \Phi^+\}$
\end{center}
should belong to $p^k\mathbb{Z}_p$ and not all of them in $p^{k+1}\mathbb{Z}_p$, because if all the elements $u_{\beta},v_{\delta},w_{\alpha} \in p^{k+1}\mathbb{Z}_p$, then $g \in G(k+1)$ and this is a contradiction to our assumption.
\end{proof}
For $m_{\beta},n_{\delta},z_{\alpha} \in \mathbb{Z}_p$ we have
\begin{align*}
&\omega\Big(\prod_{{\beta} \in \Phi^-}x_{\beta}(p)^{m_{\beta}}\prod_{{\delta} \in \Pi}h_{\delta}(1+p)^{n_{\delta}}\prod_{{\alpha} \in \Phi^+}x_{\alpha}(p)^{z_{\alpha}}\Big)\\
&=\omega\Big(\prod_{{\beta} \in \Phi^-}x_{\beta}(pm_{\beta})\prod_{{\delta} \in \Pi}h_{\delta}((1+p)^{n_{\delta}})\prod_{{\alpha} \in \Phi^+}x_{\alpha}(pz_{\alpha})\Big)\\
&=\min_{{\beta},{\alpha},{\delta}}\{1+val_p(m_{\beta}),1+val_p(n_{\delta}),1+val_p(z_{\alpha})\}.
\end{align*}
The first equality follows from p.$30$, and lemma $28$ of \cite{Yale} and the second equality follows from Theorem $\ref{eq:th1.3.3.3}$. 

Since  $\omega(x_{\beta}(p))=\omega(x_{\alpha}(p))=\omega(h_{\delta}(1+p))=1$ and $p>2$, $G(1)$ is also $p$-saturated (cf. $\ref{eq:saturated}$). Hence, with  theorem $\ref{triangulard}$, it follows that:
\begin{theorem}
\label{eq:orderedbasisthm}
The first congruence kernel $G(1)$ is $p$-saturated with an ordered basis 

\begin{center}
$\{x_{\beta}(p),h_{\delta}(1+p),x_{\alpha}(p);{\beta} \in \Phi^-,{\delta} \in \Pi, {\alpha} \in \Phi^+\}$,
\end{center} 
the order being given by the height function on the roots.
\end{theorem}

The group $G(1)$ has a decomposition
\begin{center}
$G(1)=\prod_{\beta \in \Phi^-}N_{\beta}\prod_{\delta \in \Pi}T_{\delta}\prod_{\alpha \in \Phi^+}N_{\alpha}$,
\end{center}

  where 
  
  \begin{align*}
  N_{\beta}&=\langle x_{\beta}(u),u\in p\mathbb{Z}_p\rangle,\\
  N_{\alpha}&=\langle x_{\alpha}(w),w \in p\mathbb{Z}_p \rangle,\\
  T_{\delta}&=\langle h_{\delta}(1+v),v\in 1+p\mathbb{Z}_p\rangle;
  \end{align*} 
  the order of the products is taken according to the height function on the roots. 
  
  Let $\Lambda (G(1))$  be the Iwasawa algebra of  $\mathbb{Z}_p$-valued measures on $G(1)$.  It can also be thought as distributions on $G(1)$ in the sense of \cite{Wash}. The  decomposition of $G(1)$ given above yields a decomposition of $\Lambda(G(1))$ as a topological $\mathbb{Z}_p$-module:
  \begin{equation}\label{eq:completedtensor}
  \Lambda(G(1))=\Lambda(N_{-\alpha_{max}}) \hat{\otimes}\cdots \hat{\otimes}\Lambda(T_{\delta}) \hat{\otimes} \cdots \hat{\otimes} \Lambda(N_{\alpha_{max}})
  \end{equation}
 where ${\alpha_{max}}$ is the highest root and the order of the product is taken according to the height function on the roots. The factors of ($\ref{eq:completedtensor}$) are the spaces of distributions on the factors of $G(1)$. If $f$ is a function on $G(1)$ and $U_{\beta},V_{\alpha},W_{\delta}$ distributions on $N_{\beta},T_{\delta},N_{\alpha}$, then
\begin{align}
\label{eq:completed2}
&<U_{-\alpha_{max}}\otimes \cdots \otimes W_{\delta} \otimes \cdots \otimes V_{\alpha_{max}},f>\\
&:=<U_{-\alpha_{max}}\otimes \cdots \otimes W_{\delta} \otimes \cdots \otimes V_{\alpha_{max}},f(u_{-\alpha_{max}}\cdots h_{\delta} \cdots n_{\alpha_{max}})>
\end{align}
where $u_{\beta} \in N_{\beta},h_{\delta} \in T_{\delta}, n_{\alpha} \in N_{\alpha}$ and $f$ is seen as a function on 
\begin{equation*}
 \prod_{\beta \in \Phi^-}N_{\beta}\prod_{\delta \in \Pi}T_{\delta}\prod_{\alpha \in \Phi^+}N_{\alpha}.
 \end{equation*}
  For each factor, $U=N_{\beta},T_{\delta}$ or $N_{\alpha}$ of $G(1)$, $\Lambda(U)$ is naturally sent to $\Lambda(G(1))$, by integrating a function $f \in C(G(1),\mathbb{Z}_p)$ against $\mu \in \Lambda(U)$ on the $U$-factor. This map is  compatible with the convolution product as in \cite{Laurent}.

\section{Alternative proof of Theorem $\ref{eq:orderedbasisthm}$}

In this section we  briefly sketch another proof of theorem $\ref{eq:orderedbasisthm}$ using group theory. Dixon, Sautoy, Mann and Segal in \cite{Analy} describe a subclass of $p$-saturated groups called uniform pro-$p$ groups. It is not clear whether any $p$-saturated group is uniform pro-$p$ (cf. notes at the end of chapter $4$ of \cite{Analy}). In this section, without giving the proofs, we  use the results of \cite{Analy} to provide another proof of theorem $\ref{eq:orderedbasisthm}$.

Let $p>2$ be a prime. Let $G$ be a pro-$p$ group which is topologically finitely generated. Then we say that $G$ is powerful if $G/\overline{G^p}$ is abelian. Moreover, if $G$ is torsion-free, then we say that $G$ is uniform. Note that \cite{Analy} has a different definition for uniform pro-$p$ group, but  Theorem $4.5$ of \cite{Analy} shows that it is equivalent to the definition given above. For uniform pro-$p$ group $G$ we have the following proposition:
\begin{proposition}\label{eq:fact1}
 Let $G$ be a topologically finitely generated uniform pro-$p$ group. Let $G=\overline{\langle a_1,\cdots a_d\rangle}$ such that $d$ is minimum. Then 
\begin{center}
$G=\overline{\langle a_1\rangle}\cdots \overline{\langle a_d\rangle}$  
\end{center}
and the mapping 
\begin{center}
$(\lambda_1,...,\lambda_d) \rightarrow a_1^{\lambda_1}\cdots a_d^{\lambda_d}$
\end{center}
from $\mathbb{Z}_p^d$ to $G$ is a homeomorphism.
\end{proposition}

\begin{proof}
See  proposition $3.7$ and theorem $4.9$ of \cite{Analy}.
\end{proof} 
\begin{remark}
\label{eq:labeluniform}
With the hypothesis as in proposition \ref{eq:fact1}, the discussion in section $4.2$ of \cite{Analy} shows that $G$ has an integrally valued $p$-valuation $\omega$ and an Lazard's ordered basis $a_1,...,a_d$ with $\omega(a_i)=1,i\in [1,d]$ (see also the remark after lemma $4.3$ of \cite{Algebras}).
\end{remark}

In section $8.2$ of \cite{Analy}, the authors describe $p$-adic analytic groups. Without giving the definition, we just like to point out that uniform pro-$p$ groups are  $p$-adic analytic groups (cf. example $5$, section $8.17$ of \cite{Analy}).

In the following we define the notion of standard groups (cf. definition $8.2.2$ of \cite{Analy}). 

\begin{definition}
 Let $G$ be a $p$-adic analytic group. Then $G$ is a standard group (of dimension $d$ over $\mathbb{Q}_p$) if
\begin{flushleft}
\label{eq:standard}
(i) the analytic manifold structure on $G$ is defined by a global atlas of the form $\{(G,\psi, d)\}$ where $\psi$ is a homeomorphism of $G$ onto $({p\mathbb{Z}_p})^d$ with $\psi(1)=0$,  \\
(ii) for $j=1,...,d$, there exists $P_j(X,Y)\in \mathbb{Z}_p[[X,Y]]$ such that $$\psi_j(xy^{-1})=P_j(\psi(x),\psi(y))$$ for all $x,y \in G$, where $\psi=(\psi_1,...,\psi_d)$.
\end{flushleft}
\end{definition}

\begin{proposition}\label{eq:fact2}
Let $G$ be a standard group of dimension $d$ over $\mathbb{Q}_p$. Then $G$ is a uniform pro-$p$ group of dimension $d$. 
\end{proposition} 

\begin{proof}
See theorem $8.31$ of \cite{Analy}. 
\end{proof}
\begin{lemma}
\label{eq:lemmadixon}
The first congruence kernel $G(1)$ is a $\mathbb{Z}_p$- standard group (of level $1$) with dimension $|\Phi|+|\Pi|$.
\end{lemma}
\begin{proof}
By theorem \ref{triangulard} each element $g \in G(1)$ can be uniquely written as 

$\prod_{\beta \in \Phi^-}x_{\beta}(u_\beta)\prod_{\delta \in \Pi}h_\delta(1+v_{\delta})\prod_{\alpha \in \Phi^+}x_{\alpha}(w_{\alpha})$ for some $u_{\beta},v_{\delta},w_{\alpha} \in p\mathbb{Z}_p$. The proof then follows from theorem $8$ of \cite{Yale} and  exercise $11$ of chapter $13$ of \cite{Analy}.
\end{proof}

\textit{\textbf{Alternative proof of Theorem $\ref{eq:orderedbasisthm}$}}: 
 Lemma \ref{eq:lemmadixon} gives us that $G(1)$ is a  $\mathbb{Z}_p$-standard group of dimension $|\Phi|+|\Pi|$. Then proposition \ref{eq:fact2} gives that $G(1)$ is a uniform pro-$p$ group. But uniform pro-$p$ groups are $p$-valuable and $p$-saturated (cf. remark after lemma $4.3$ of \cite{Algebras}.) Now, 
 Theorem $\ref{triangulard}$ shows that 
 \begin{center}
 $G(1)=\prod_{{\beta} \in \Phi^-}\overline{x_{\beta}(p)}\prod_{{\delta} \in \Pi}\overline{h_{\delta}(1+p)}\prod_{{\alpha} \in \Phi^+}\overline{x_{\alpha}(p)}$,
 \end{center} 
 where the bars denotes the closure, that is 
 \begin{align*}
  \overline{x_{\beta}(p)}&=\{x_{\beta}(p)^{m_{\beta}}, \forall m_{\beta} \in \mathbb{Z}_p\}=\{x_{\beta}(m_{\beta}p), \forall m_{\beta} \in \mathbb{Z}_p\}, \\
  \overline{x_{\alpha}(p)}&=\{x_{\alpha}(p)^{z_{\alpha}}, \forall z_{\alpha} \in \mathbb{Z}_p\}=\{x_{\alpha}(z_{\alpha}p), \forall z_{\alpha} \in \mathbb{Z}_p\}, \\
  \overline{h_{\delta}(1+p)}&=\{h_{\delta}(1+p)^{n_{\delta}}, \forall n_{\delta} \in \mathbb{Z}_p\}=\{h_{\delta}((1+p)^{n_{\delta}}), \forall n_{\delta} \in \mathbb{Z}_p\}.  
 \end{align*}

 Then remark $\ref{eq:labeluniform}$ finishes the proof of Theorem $\ref{eq:orderedbasisthm}$ [Q.E.D]. 

\section{Iwasawa algebras and relations}

In this section, for $\alpha \in \Phi$ and $\delta \in \Pi$, we identify (as $\mathbb{Z}_p$-module) the Iwasawa algebra of $G(1)$ with the ring of power series in the variables $V_{\alpha}$ and $W_{\delta}$ over $\mathbb{Z}_p$ (cf. \ref{eq:5.1clozel}). This isomorphism is given by sending $1+V_{\alpha} \mapsto x_{\alpha}(p)$ and $1+W_{\delta} \mapsto h_{\delta}(1+p)$. As the Iwasawa algebra is non-commutative, this identification is obviously not a ring isomorphism. Therefore, the goal of this section is to study the product of the variables in wrong order, viewing them as elements of the Iwasawa algebra, and then find the relations among the variables $(\ref{eq:firstrelation}-\ref{eq:forthrelation})$.\\ 
 
 Theorem $\ref{eq:orderedbasisthm}$ gives us an ordered basis of $G(1)$ with the $p$-valuation  $\omega$.  By definition we have a homeomorphism
\begin{center}
$c:\mathbb{Z}_p^d\rightarrow G(1)$\\
$(y_1,...,y_d)\mapsto g_1^{y_1}\cdots g_d^{y_d}$
\end{center}
where $d=|\Phi|+|\Pi|$ and $(g_1,...,g_d)$ is the ordered basis of $G(1)$. Let $C(G(1))$ be continuous functions from $G(1)$ to $\mathbb{Z}_p$. The map $c$ induces, by pulling back functions, an isomorphism of $\mathbb{Z}_p$-modules
\begin{center}
$c^*:C(G(1)) \simeq C(\mathbb{Z}_p^d)$.
\end{center}

We may recall that $\Lambda (G(1))$  is the Iwasawa algebra of $G(1)$ over $\mathbb{Z}_p$ introduced in section $3$ that is $\Lambda (G(1)):=\varprojlim _{N\in \mathcal{N}(G(1))}(G(1)/N)$ where $\mathcal{N}(G(1))$ is the set of open normal subgroups in $G(1)$.  Lemma 22.1 of \cite{Lie} shows that 
 \begin{center}
 $\Lambda (G(1))=Hom_{\mathbb{Z}_p}(C(G(1)),\mathbb{Z}_p)$.
  \end{center}
  So dualizing the map $c^*$, we get an isomorphism of $\mathbb{Z}_p$-modules 
  \begin{center}
  $c_{*}=\Lambda (\mathbb{Z}_p^d)\cong \Lambda (G(1))$.
  \end{center}
  This gives us the following isomorphism of $\mathbb{Z}_p$-modules:
  \begin{align}\label{eq:5.1clozel}
  \tilde{c}:\mathbb{Z}_p[[V_{\alpha},W_{\delta}; \alpha \in \Phi, {\delta} \in \Pi]]&\cong \Lambda(G(1))\\
  1+V_{\alpha}&\mapsto x_{\alpha}(p)\\
  1+W_{{\delta}}& \mapsto h_{{\delta}}(1+p).
  \end{align}
  This is because the Iwasawa algebra of $\mathbb{Z}_p^d$ can be identified with the ring of power series in $d$ variables (cf. prop $20.1$ of  \cite{Lie}). From ($\ref{eq:completedtensor}$),  any $\lambda \in \Lambda(G(1))$ can be written uniquely as 
  \begin{equation}\label{eq:iwasawaequation}
  \lambda =\sum_m \lambda_mV_{-\alpha_{max}}^{m_1}\cdots W_{\delta}^{m_{-}} \cdots V_{\alpha_{max}}^{m_d},
  \end{equation}
  with $\lambda_m \in \mathbb{Z}_p,m=(m_1,...,m_d)\in \mathbb{N}^d$. The product of the above variables 
  \begin{center}
  $V_{-\alpha_{max}}^{m_1}\cdots W_{\delta}^{m_{-}} \cdots V_{\alpha_{max}}^{m_d}:=V_{-\alpha_{max}}^{m_1}\otimes \cdots \otimes W_{\delta}^{m_{-}} \otimes \cdots \otimes V_{\alpha_{max}}^{m_d}$
  \end{center} 
  is taken according to the height function on the roots. We note that we can write unambiguously the variables $V_{\alpha}^{n},W_{\delta}^n;n \geq 0$ because from the discussion after ($\ref{eq:completed2}$) in section $3$, it follows that the convolution of $V_{\alpha}$'s (or  $W_{\delta}$'s) taken $n$ times equals their tensor product, i.e. 
  \begin{center}
  $V_{\alpha}^n=V_{\alpha} \otimes \cdots \otimes V_{\alpha}=V_{\alpha} \ast \cdots \ast V_{\alpha}$.
  \end{center}
   It immediately follows from ($\ref{eq:completed2}$) that the convolution (taken on  $G(1)$) of the distributions $V_{\alpha},W_{\delta}$ equals their tensor product when they are taken in the order of the height function. We simply write, consistent with our previous notation, 
  \begin{center}
  $V_{-\alpha_{max}}^{m_1}\cdots W_{\delta}^{m_{-}} \cdots V_{\alpha_{max}}^{m_d}=V_{-\alpha_{max}}^{m_1}\ast \cdots \ast W_{\delta}^{m_{-}}\ast  \cdots \ast  V_{\alpha_{max}}^{m_d}$,
  \end{center}
  where 
  \begin{equation}
  \label{eq:Iwasawa2}
   V_{\alpha}^{n}=V_{\alpha} \ast V_{\alpha} \ast \cdots \ast V_{\alpha},(n \geq 0)
  \end{equation}
  
  and 
  \begin{equation}
  \label{eq:Iwasawa3}
  W_{\delta}^n=W_{\delta}\ast W_{\delta}\ast \cdots \ast W_{\delta},(n \geq 0).
  \end{equation}

  In the following we study the product of the variables in the wrong order.
  
   Let $b_i:=g_i-1$, $(g_1,...,g_d)$ being the Lazard ordered basis of $G(1)$ and $b^m:=b_1^{m_1}\cdots b_d^{m_d}$, for any multi-index $m=(m_1,...,m_d)\in \mathbb{N}^d$, in $\Lambda(G(1))$. We define a function 
  \begin{center}
  $\tilde{\omega}:\Lambda(G(1))\backslash \{0\}\rightarrow [0,\infty)$
  \end{center}
  by
  \begin{center}
  $\tilde{\omega}\Big(\sum_mc_mb^m\Big):=\inf_m \Big(val_p(c_m)+|m|\Big)$
  \end{center}
  where $c_m \in \mathbb{Z}_p,|m|:=m_1+\cdots + m_l$, with the convention that $\tilde{\omega}(0):=\infty$. 
  
 By the isomorphism $\tilde{c}$, we can identify the variables $V_{\alpha}, W_{{\delta}};(\alpha \in \Phi,\delta \in \Pi)$ as elements of the Iwasawa algebra of $G(1)$. Our goal is to find the  relations among the above variables. For this, we use the Chevalley relations given by Steinberg in \cite{Yale}. With $x_{\alpha}(p)$, $h_{\delta}(1+p)$ as defined in section $2$, \cite{Yale} gives
 
 \begin{equation*}
h_{{\delta}}(1+p)x_{\alpha}(p)h_{{\delta}}(1+p)^{-1}=x_{\alpha}({(1+p)}^{\langle \alpha , {\delta} \rangle }p),(\alpha \in \Phi,\delta \in \Pi)
 \end{equation*}
 where $ \langle \alpha , {\delta} \rangle=2(\alpha,{\delta})/({\delta},{\delta}) \in \mathbb{Z}$ (cf.   \cite{Yale}, p.$30$). It follows that the corresponding relation in the Iwasawa algebra of $G(1)$ is 
 \begin{equation}\label{eq:firstrelation}
 (1+W_{{\delta}})(1+V_{\alpha})={(1+V_{\alpha})}^{(1+p)^{\langle \alpha , {\delta} \rangle}}(1+W_{{\delta}}),(\alpha \in \Phi,\delta \in \Pi).
 \end{equation}
 Let $\alpha_1,\alpha_2 \in \Phi, \alpha_1 \neq -\alpha_2 $ and $\alpha_1+\alpha_2 \notin \Phi$ then example (a), p.$24$ of \cite{Yale} gives 
 \begin{equation*}
 x_{\alpha_1}(p)x_{\alpha_2}(p)=x_{\alpha_2}(p)x_{\alpha_1}(p),(\alpha_1,\alpha_2 \in \Phi).
 \end{equation*}
  Thus, the relation in the Iwasawa algebra is 
  \begin{equation}\label{eq:secondequation}
  (1+V_{\alpha_1})(1+V_{\alpha_2})=(1+V_{\alpha_2})(1+V_{\alpha_1}),(\alpha_1,\alpha_2 \in \Phi).
  \end{equation}
  If $\alpha_1\neq -\alpha_2$ and $\alpha_1+\alpha_2 \in \Phi$ then $x_{\alpha_1}(p)x_{\alpha_2}(p)=[x_{\alpha_1}(p),x_{\alpha_2}(p)]x_{\alpha_2}(p)x_{\alpha_1}(p)$, where $[,]$ is the commutator function. So lemma $15$, p.$22$ of \cite{Yale} gives
  \begin{equation*}
  x_{\alpha_1}(p)x_{\alpha_2}(p)=\Big(\prod_{i,j>0}x_{i\alpha_1+j\alpha_2}(c_{ij}p^ip^j)\Big)x_{\alpha_2}(p)x_{\alpha_1}(p),(\alpha_1,\alpha_2 \in \Phi),
  \end{equation*}
  where $c_{ij} \in \mathbb{Z}$ and the order of the product is as prescribed in lemma $15$ of \cite{Yale}. This gives
  \begin{equation}\label{eq:thirdequation}
(1+V_{\alpha_1})(1+V_{\alpha_2})=\Big(\prod_{i,j>0}(1+V_{i\alpha_1+j\alpha_2})^{c_{ij}p^{i+j-1}}\Big)(1+V_{\alpha_2})(1+V_{\alpha_1}),(\alpha_1,\alpha_2 \in \Phi).
  \end{equation}
If $\alpha_3 \in \Phi^+$, corollary $6$, p.$46$ of \cite{Yale} gives a homomorphism
\begin{center}
$\varphi_{\alpha_3}:SL_2(\mathbb{Q}_p) \rightarrow \langle x_{\alpha_3}(t),x_{-\alpha_3}(t);t \in \mathbb{Q}_p\rangle $
\end{center}
 such that
 \begin{align*}
 \varphi_{\alpha_3}(1+tE_{12})&=x_{\alpha_3}(t),\\
 \varphi_{\alpha_3}(1+tE_{21})&=x_{-\alpha_3}(t),\\
 \varphi_{\alpha_3}(tE_{11}+t^{-1}E_{22})&=h_{\alpha_3}(t),
 \end{align*}
  where $E_{ij}$ is the $2 \times 2$ matrix with $1$ in the $(i,j)^{th}$ entry and zero elsewhere. We have 
\[ \left( \begin{array}{cc}
	1 & p \\
	0 & 1
	\end{array} \right)
	\left( \begin{array}{cc}
	1 & 0 \\
	p & 1
	\end{array} \right)
	=
	\left( \begin{array}{cc}
	1 & 0 \\
	v & 1
	\end{array} \right)
	\left( \begin{array}{cc}
	t & 0 \\
	0 & t^{-1}
	\end{array} \right)
	\left( \begin{array}{cc}
	1 & u \\
	0 & 1
	\end{array} \right)
	\]
where $t=1+p^2,u=p(1+p^2)^{-1},v=p(1+p^2)^{-1}$. As $\varphi_{\alpha_3}$ is a homomorphism,  we therefore obtain
\begin{equation*}
x_{\alpha_3}(p)x_{-\alpha_3}(p)=x_{-\alpha_3}(v)h_{\alpha_3}(t)x_{\alpha_3}(u),(\alpha_3 \in \Phi^+).
\end{equation*}
As before let $\delta_1,...,\delta_l$ be the simple roots where $l=|\Pi|$. Then, as in the proof of Theorem $\ref{triangulard}$, we can uniquely decompose
\begin{center}
$h_{\alpha_3}(t)=\prod_{i=1}^lh_{\delta_i}(t^{n_i})$ for $n_i \in \mathbb{Z}$.
\end{center}
Now, let $P=\frac{log(1+p^2)}{log(1+p)} \in \mathbb{Z}_p, 1+p^2=(1+p)^P$ and let $Q=(1+p^2)^{-1}$ then 
\begin{equation*}
h_{\alpha_3}(t)=h_{\alpha_3}(1+p)^P=\prod_{i=1}^lh_{\delta_i}(1+p)^{n_iP},(\alpha_3 \in \Phi^+,\delta_i \in \Pi).
\end{equation*}
This gives that the corresponding relation in the Iwasawa algebra is 
\begin{equation}\label{eq:forthrelation}
(1+V_{\alpha_3})(1+V_{-\alpha_3})=(1+V_{-\alpha_3})^Q\Big(\prod_{i=1}^l(1+W_{\delta_i})^{n_iP}\Big) (1+V_{\alpha_3})^Q,(\alpha_3 \in \Phi^+,\delta_i \in \Pi).
\end{equation}
Thus we have shown 
\begin{lemma}
With the notations as above, we have the following relations in $\Lambda(G(1))$:
\begin{enumerate}
\item $(1+W_{{\delta}})(1+V_{\alpha})={(1+V_{\alpha})}^{(1+p)^{\langle \alpha , {\delta} \rangle}}(1+W_{{\delta}})$,\\
\item $(1+V_{\alpha_1})(1+V_{\alpha_2})=(1+V_{\alpha_2})(1+V_{\alpha_1}),(\alpha_1+\alpha_2 \notin \Phi)$,\\
\item ${(1+V_{\alpha_1})(1+V_{\alpha_2})=\Big(\prod_{i,j>0}(1+V_{i\alpha_1+j\alpha_2})^{c_{ij}p^{i+j-1}}\Big)(1+V_{\alpha_2})(1+V_{\alpha_1}),  (\alpha_1+\alpha_2 \in \Phi),}$\\
\item $(1+V_{\alpha_3})(1+V_{-\alpha_3})=(1+V_{-\alpha_3})^Q\Big(\prod_{i=1}^l(1+W_{\delta_i})^{n_iP}\Big) (1+V_{\alpha_3})^Q$.
\end{enumerate}
\end{lemma}
Consider $\mathcal{A}$ - the universal non-commutative $p$-adic algebra in variables $V_{\alpha}$ and $W_{{\delta}}$, where $\alpha$ varies over the roots and ${\delta}$ varies over the simple roots and the ordering of the variables are given by the height function on the roots to which they correspond.  Thus it
is composed of all non-commutative series 
\begin{center}
$f=\sum_{n\geqslant 0}\sum_ia_ix^i$
\end{center}
where $a_i \in \mathbb{Z}_p$, $x^i:=x_{i(1)}\cdots x_{i(n)}$ and for all $n \geqslant 0$, $i$ runs over all maps
$\{1,2,...,n\}\rightarrow \{1,2,...,d\}$; the monomials $x_i$, with $i$ increasing, are assigned to  the product of the variables among $\{V_{\alpha},W_{{\delta}};\alpha \in \Phi,{\delta} \in \Pi \}$ corresponding to a fixed order compatible with the partial order given by the height function on the roots. The algebra $\mathcal{A}$ has a maximal ideal $\mathcal{M}_{\mathcal{A}}$ generated by $(p,x_1,...,x_d)$ and a prime ideal $\mathcal{P}_{\mathcal{A}}$ generated by $(x_1,...,x_d)$. The topology on $\mathcal{A}$ is given by the powers of $\mathcal{M}_{\mathcal{A}}$. 

Let $\mathcal{R}$ be the closed two-sided ideal generated in $\mathcal{A}$ by the relations
$(\ref{eq:firstrelation}-\ref{eq:forthrelation})$. 
 Let $\overline{\mathcal{A}}$ be the image of the reduction modulo $p$ map on $\mathcal{A}$; i.e. we consider the non-commutative series with coefficients in $\mathbb{F}_p$ with its natural topology given by its maximal ideal $\mathcal{M}_{\overline{\mathcal{A}}}$. 
\begin{lemma}
\label{eq:lemmaclosure}
 Let $\overline{\mathcal{R}}$ be the image of $\mathcal{R}$ in $\overline{\mathcal{A}}$. Then $\overline{\mathcal{R}}$ is the closed two-sided ideal generated in $\overline{\mathcal{A}}$ by the images of the relations $(\ref{eq:firstrelation},\ref{eq:secondequation},\ref{eq:thirdequation},\ref{eq:forthrelation})$.
\end{lemma}
\begin{proof}
We follow the proof exactly as in lemma $1.3$ of \cite{Laurent}. For completeness we repeat the proof. We denote by $\mathcal{I} \subset \mathcal{A}$ the \textit{ideal} generated by the relations, let $\mathcal{J} \subset \overline{\mathcal{A}}$ be the similar ideal. Then $\mathcal{J}$ is obviously the image of $\mathcal{I}$ in $\overline{\mathcal{A}}$; we denote it by $\overline{\mathcal{I}}$.

Let $\overline{\mathcal{R}}$ be the reduction of $\mathcal{R}$, and consider the closure $cl(\overline{\mathcal{I}})$ of $\overline{\mathcal{I}}$ in $\overline{\mathcal{A}}$. If $f \in \mathcal{R}$, we have $f=\lim_nf_n(f_n \in \mathcal{I})$ for the topology given by $\mathcal{M}_{\mathcal{A}}^N.$ This implies that $\overline{f}=\lim \overline{f_n}$ for the topology given by $\mathcal{M}_{\overline{\mathcal{A}}}^N$ on $\overline{\mathcal{A}}$, thus $\overline{f}\in cl(\overline{\mathcal{I}}).$ Conversely, assume $\overline{f} \in \overline{\mathcal{A}}$ can be written $\overline{f}=\lim \overline{f_n}$ with $\overline{f_n} \in \overline{\mathcal{I}}$. Then, $\overline{f_n}$ is the reduction of a series $f_n \in \mathcal{I} \subset \mathcal{R}$. Since $\mathcal{R}$ is closed and $\mathcal{A}$ is compact, we may assume that $f_n$ converges to $g \in \mathcal{R}$. Then, by definition of the topologies, $\overline{f}=\lim \overline{f_n}=\overline{g}$. Thus $cl(\overline{\mathcal{I}})=\overline{\mathcal{R}}$, which finishes the proof.
\end{proof}

\section{Presentation of the Iwasawa algebra $\Lambda(G(1))$ for $p>2$}
Our goal in this section is to give an explicit presentation of $\Lambda(G(1))$ (theorem \ref{eq:maintheorem}). The strategy of the proof is to first show the corresponding statement with coefficients in $\mathbb{F}_p$ and then lift the coefficients to $\mathbb{Z}_p$.  Let $$\Omega_{G(1)}=\Lambda(G(1))\otimes_{\mathbb{Z}_p}\mathbb{F}_p$$ be the Iwasawa algebra with finite coefficients as in \cite{Laurent}. We show in theorem \ref{eq:secondimptheorem} that 
for $p>2$, the Iwasawa algebra mod $p$, $\Omega_{G(1)}$, is naturally isomorphic to $\overline{\mathcal{A}}/\overline{\mathcal{R}}$. We first construct a map $\overline{\varphi}:\overline{\mathcal{B}}=\overline{\mathcal{A}}/\overline{\mathcal{R}}\rightarrow \Omega_{G(1)}$ (see corollary $\ref{eq:corollaryimp}$), then using the natural grading (see the discussion before proposition $\ref{eq:propositiononly}$) on $\overline{\mathcal{B}}$, we show that $\dim$ \text{gr}$^n\overline{\mathcal{B}}\leqslant\dim S_n$ (see proposition $\ref{eq:propositiononly}$), where $S_n$ is the space of homogeneous commutative polynomials in the variables $\{V_{\alpha},W_{\delta};\alpha \in \Phi, \delta \in \Pi\}$ over $\mathbb{F}_p$ of degree $n$.\\

The discussion before Theorem $\ref{eq:maintheorem}$ gives up a 
 natural map $\varphi$ from $\mathcal{A} \rightarrow \Lambda(G(1))$. The map $\varphi$ sends $x_i (1 \leq i \leq d;i\ \text{increasing})$ to the  variable $\{V_{\alpha},W_{\delta};\alpha \in \Phi, \delta \in \Pi\}$ corresponding to the order compatible with the height function on the roots.  Let us define
\begin{center}
$\mathcal{M}_{\Lambda}^N:=\{\lambda \in \Lambda(G(1))|\tilde{\omega}(\lambda)\geqslant N\}$.
\end{center}
It follows from Lazard's results \cite{Lazard} (also Schneider, chapter $6$ of \cite{Lie} ) that $\mathcal{M}_{\Lambda}^N$ is indeed the $N^{th}$ power of the maximal ideal $\mathcal{M}_{\Lambda}$ of $\Lambda(G(1))$. 

We consider the filtration of $\mathcal{A}$ given by the powers of its maximal ideal. Then we can define a valuation $w_{\mathcal{A}}$ given by:
\begin{center}
let $f=\sum_ia_ix^i$,\\
then $w_{\mathcal{A}}(f)=\inf_i(val_p(a_i)+|i|)$
\end{center}
where $|i|=n$ is the degree of $i$. 
\begin{lemma}\label{eq:1.6.1clozel}
The natural map $\varphi: \mathcal{A} \rightarrow \Lambda(G(1))$ is surjective and continuous. Moreover for $N\geqslant 0$, we have 
\begin{center}
$\varphi(\mathcal{M}_{\mathcal{A}}^N)\subset \mathcal{M}_{\Lambda}^N$.
\end{center}
\end{lemma}
\begin{proof}
Surjectivity of the map $\varphi$ is obvious from $(\ref{eq:iwasawaequation}-\ref{eq:Iwasawa3})$. Now, for continuity we use the fact that the valuation $\tilde{\omega}$ is additive i.e.
\begin{center}
$\tilde{\omega}(\lambda * \mu)=\tilde{\omega}(\lambda)+\tilde{\omega}(\mu); \lambda,\mu \in \Lambda(G(1))$.
\end{center}
This is a nontrivial fact proved by Lazard (cf. \cite{Lazard}, III.$2.3.3$ ).

The continuity of the map $\varphi$ is implied by the stronger property
\begin{equation}\label{eq:continuity}
\tilde{\omega}(\varphi(x^i))=n=|i|
\end{equation}
where $n$ is the degree of the monomial. By induction on $n$, this follows from the non-trivial fact that $\tilde{\omega}$ is additive. If $f \in \mathcal{M}_{\mathcal{A}}^N$, we have $w_{\mathcal{A}}(f)\geqslant N$ and the continuity follows from $\ref{eq:continuity}$ by $\mathbb{Z}_p$-linearity. This completes the proof.
\end{proof}
This gives the following two corollaries:
\begin{corollary}
There is a natural continuous surjection
\begin{center}
$\mathcal{B}=\mathcal{A}/\mathcal{R} \rightarrow \Lambda(G(1))$.
\end{center}
\end{corollary}
\begin{corollary}
\label{eq:corollaryimp}
There is a continuous surjection
\begin{center}
$\overline{\varphi}:\overline{\mathcal{B}}=\overline{\mathcal{A}}/\overline{\mathcal{R}}\rightarrow \Omega_{G(1)}.$
\end{center}
\end{corollary}
This follows from Lemma $\ref{eq:lemmaclosure}$. 

By Abelian distribution theory, $\Omega_{G(1)}$ is, as a space, isomorphic to the commutative power series ring in the variables $V_{\alpha}$ and $W_{{\delta}}$ over $\mathbb{F}_p$  with the compact topology, where $\alpha$ varies over the roots and ${\delta}$ varies over the simple roots. By obvious computation one can show that 
\begin{center}
$\mathcal{M}_{\Omega}^N=\{\lambda \in \Omega_{G(1)}:v_{\Omega}(\lambda)\geqslant N\}$,
\end{center}
$v_{\Omega}$ being the usual valuation on power series, is the reduction of $\mathcal{M}_{\Lambda}^N.$ It is also a (two-sided) ideal, $ \mathcal{M}_{\Omega}$ being the maximal ideal. 

Similarly, in $\mathcal{A}$ we have that the reduction mod $p$ of $\mathcal{M}_{\mathcal{A}}^N$ is the ideal of series 
\begin{center}
$\overline{f}=\sum_ia_ix^i$ with $(a_i \in \mathbb{F}_p)$
\end{center}
such that $|i|\geqslant N$. We obtain the maximal ideal of $\overline{\mathcal{A}}$ by setting $N=1$. Furthermore, we have $(\mathcal{M}_{\overline{\mathcal{A}}})^N=\mathcal{M}_{\overline{\mathcal{A}}}^N.$

In the following we  study  the algebra $\overline{\mathcal{B}}$ using the relations 
$(\ref{eq:firstrelation},\ref{eq:secondequation},\ref{eq:thirdequation},\ref{eq:forthrelation})$ which is used to prove  Proposition $\ref{eq:propositiononly}$ below. Then, we will use it to give the proof of Theorem $\ref{eq:secondimptheorem}$, which in turn, after lifting coefficients to $\mathbb{Z}_p$, leads to the proof of Theorem $\ref{eq:maintheorem}$.

Consider the natural filtration of $\overline{\mathcal{A}}$ by the powers of $\mathcal{M}_{\overline{\mathcal{A}}}$, which we denote by $F^n\overline{\mathcal{A}}$ as in \cite{Laurent}. We have $F^n\overline{\mathcal{A}}/F^{n+1}\overline{\mathcal{A}}=$ gr$^n{\overline{\mathcal{A}}}.$ The filtration $F^n$ induces a filtration on $\overline{\mathcal{B}}$:
\begin{center}
$F^n\overline{\mathcal{B}}=F^n\overline{\mathcal{A}}+\overline{\mathcal{R}}$
\end{center}
and hence a graduation
\begin{center}
gr$^n\overline{\mathcal{B}}=F^n\overline{\mathcal{A}}+\overline{\mathcal{R}}/F^{n+1}\overline{\mathcal{A}} +\overline{\mathcal{R}}$.
\end{center}
Hence, we have 
\begin{center}
gr$^n\overline{\mathcal{B}}=F^n\overline{\mathcal{A}}/F^{n+1}\overline{\mathcal{A}}+(F^n\overline{\mathcal{A}}\cap \overline{\mathcal{R}}).$
\end{center}
Let $S_n$ be the space of homogeneous commutative polynomials in the variables $\{V_{\alpha},W_{\delta};\alpha \in \Phi, \delta \in \Pi\}$ over $\mathbb{F}_p$ of degree $n$. Let $\Sigma_n$ be the corresponding space of homogeneous non-commutative polynomials of degree $n$; hence $\Sigma_n\rightarrow F^n\overline{\mathcal{A}}/F^{n+1}\overline{\mathcal{A}}, $ and therefore $ \Sigma_n\rightarrow $ gr$^n\overline{\mathcal{B}}$, is surjective. 

\begin{proposition}
\label{eq:propositiononly}
We have $\dim$ \text{gr}$^n\overline{\mathcal{B}}\leqslant\dim S_n$.
\end{proposition}
\begin{proof}
The case $n=1$ is obvious. We  first prove the proposition for $n=2$ by studying the relations defining $\mathcal{R}$ (or rather $\overline{\mathcal{R}}$) and then we  use  an induction argument to complete the proof.

 For $\alpha \in \Phi$, $\delta \in \Pi$, relation $\ref{eq:firstrelation}$ is
\begin{center}
$(1+W_{{\delta}})(1+V_{\alpha})={(1+V_{\alpha})}^{(1+p)^{\langle \alpha , {\delta} \rangle}}(1+W_{{\delta}}),(\alpha \in \Phi,\delta \in \Pi).$
\end{center}

We set $q=(1+p)^{\langle \alpha , {\delta} \rangle}$, so that $q\equiv 1[p]$. Expanding the above relation we obtain
\begin{center}
$1+W_{{\delta}}+V_{\alpha}+W_{{\delta}}V_{\alpha}=\Big(1+qV_{\alpha}+\frac{q(q-1)}{2}V_{\alpha}^2+\cdots\Big)(1+W_{{\delta}})$.
\end{center}
As $\frac{q(q-1)}{2}\equiv 0[p]$ and $p>2$,
\begin{center}
$1+W_{{\delta}}+V_{\alpha}+W_{{\delta}}V_{\alpha}=(1+qV_{\alpha})(1+W_{{\delta}})+R(V_{\alpha})(1+W_{{\delta}})$
\end{center}
where $R(V_{\alpha})$ has degree $\geqslant 3$. Thus
\begin{center}
$W_{{\delta}}V_{\alpha}=(q-1)V_{\alpha}+q(V_{\alpha}W_{{\delta}})+R_1(V_{\alpha},W_{{\delta}})$
\end{center}
where $R_1(V_{\alpha},W_{{\delta}})$ has degree $\geqslant 3$. Since $q\equiv 1[p]$ we deduce 
\begin{center}
$W_{{\delta}}V_{\alpha}=V_{\alpha}W_{{\delta}}$ in gr$^2\overline{\mathcal{B}}$.
\end{center}
Relation $\ref{eq:secondequation}$ obviously gives for $\alpha_1,\alpha_2 \in \Phi,\alpha_1 \neq -\alpha_2,\alpha_1+\alpha_2 \notin \Phi$,
\begin{center}
$V_{\alpha_1}V_{\alpha_2}=V_{\alpha_2}V_{\alpha_1},(\alpha_1,\alpha_2 \in \Phi)$. 
\end{center}
Relation $\ref{eq:thirdequation}$ is for $\alpha_1,\alpha_2 \in \Phi,\alpha_1 \neq -\alpha_2,\alpha_1+\alpha_2 \in \Phi$,
\begin{center}
$(1+V_{\alpha_1})(1+V_{\alpha_2})=\Big(\prod_{i,j>0}(1+V_{i\alpha_1+j\alpha_2})^{c_{ij}p^{i+j-1}}\Big)(1+V_{\alpha_2})(1+V_{\alpha_1}),$
\end{center}
where $c_{ij} \in \mathbb{Z}$.
Now,
\begin{center}
$(1+V_{i\alpha_1+j\alpha_2})^{c_{ij}p^{i+j-1}}=1+cV_{i\alpha_1+j\alpha_2}+\frac{c(c-1)}{2}V_{i\alpha_1+j\alpha_2}^2+\cdots $
\end{center}
where $c=c_{ij}p^{i+j-1}$. It is easy to see that $p|c$ if $i>0$ and $j>0$. As $p>2$, and $2$ is invertible in $\mathbb{Z}_p$ we deduce
\begin{center}
$(1+V_{i\alpha_1+j\alpha_2})^{c_{ij}p^{i+j-1}}\equiv 1[\mod F^3\overline{\mathcal{B}}]$.
\end{center}
This implies
\begin{center}
$V_{\alpha_1}V_{\alpha_2}=V_{\alpha_2}V_{\alpha_1}$ in gr$^2\overline{\mathcal{B}}$. 
\end{center}
Relation $\ref{eq:forthrelation}$ for $\alpha_3 \in \Phi^+$ is the following:
\begin{center}
$(1+V_{\alpha_3})(1+V_{-\alpha_3})=(1+V_{-\alpha_3})^Q\Big(\prod_{i=1}^l(1+W_{\delta_i})^{n_iP}\Big) (1+V_{\alpha_3})^Q,(\alpha_3 \in \Phi^+,\delta_i \in \Pi).$
\end{center}
The constant $Q=(1+p^2)^{-1}\equiv 1[p^2], \frac{Q(Q-1)}{2}\equiv 0[p]$. As $P=\frac{log(1+p^2)}{log(1+p)}\equiv p[p^2]$, $p>2$,  $2$ is invertible in $\mathbb{Z}_p$, we get that 
\begin{center}
$(1+W_{\delta_i})^P\equiv 1[\mod (p,W_{\delta_i}^3)],(\delta_i \in \Pi)$.
\end{center}
Hence, modulo $F^3\overline{\mathcal{B}}$, relation  $\ref{eq:forthrelation}$ reduces to
\begin{center}
$(1+V_{\alpha_3})(1+V_{-\alpha_3})=1+QV_{-\alpha_3}+QV_{\alpha_3}+Q^2V_{-\alpha_3}V_{\alpha_3}[\mod F^3\overline{\mathcal{B}}],(\alpha_3 \in \Phi^+)$.
\end{center}
Therefore, as $Q\equiv 1[p^2]$ we obtain 
\begin{center}
$V_{\alpha_3}V_{-\alpha_3}\equiv V_{-\alpha_3}{V_{\alpha_3}}$ in gr$^2\overline{\mathcal{B}}$ for $\alpha_3 \in \Phi^+$.
\end{center}
This proves
 \begin{center}
$\dim$ gr$^2\overline{\mathcal{B}}\leqslant\dim S_2$.
\end{center}

Now, consider an arbitrary monomial of degree $n$,
\begin{center}
$x^i=x_{i_1}...x_{i_n}$.
\end{center}

As in Lemma $3.2$ of \cite{Laurent}, we can change $x^i$ into a well-ordered monomial ($b \mapsto i_b$ increasing) by a sequence of transpositions. Consider a move $(b,b+1)\mapsto (b+1,b)$ and assume $i_b>i_{b+1}$. We write 
\begin{center}
$x^i=x^fx_{i_b}x_{i_{b+1}}x^e$
\end{center}
where deg$(f)=r$, deg$(e)=s$ and deg$(i)=n$. Then, from the proof of Proposition $\ref{eq:propositiononly}$ we have $x_{i_b}x_{i_{b+1}}=x_{i_{b+1}}x_{i_b}[F^3\overline{\mathcal{B}}]$. Hence, we obtain $x^fx_{i_{b+1}}x_{i_b}x^e\equiv x^i[F^{n+1}], n=r+s+2$. This reduces the number of inversions in $x^i$. 

This completes the proof of proposition $\ref{eq:propositiononly}$. 
\end{proof}

In the following we state our main theorems of this paper. After proposition \ref{eq:propositiononly} their proofs directly follow from \cite{Laurent}. But for convenience of the reader we include the proofs.

\begin{theorem}
\label{eq:secondimptheorem}
For $p>2$, the Iwasawa algebra mod $p$, $\Omega_{G(1)}$, is naturally isomorphic to $\overline{\mathcal{A}}/\overline{\mathcal{R}}$.
\end{theorem}

\begin{proof}
  The natural map $\varphi: \mathcal{A} \rightarrow \Lambda(G(1))$ sends $\mathcal{M}_{\mathcal{A}}^n$ to $\mathcal{M}_{\Lambda}^n$. As $F^{\bullet}$ on $\overline{\mathcal{B}}$ is the filtration inherited from the natural filtration on $\overline{\mathcal{A}}$, we see that $\overline{\varphi}$ sends $F^n\overline{\mathcal{B}}$ to $\mathcal{M}_{\Omega}^n$.  As $\overline{\varphi}$ is surjective, the natural map 
\begin{center}
gr$\overline{\varphi}:$ gr$^{\bullet}\overline{\mathcal{B}}\rightarrow $ gr$^{\bullet}\Omega_{G(1)}$
\end{center}
is surjective. Moreover, it is an isomorphism because $\dim $ gr$^n\overline{\mathcal{B}}\leqslant \dim S_n =\dim $ gr$^n\Omega_{G(1)}$. (The last equality follows from the discussion after corollary $\ref{eq:corollaryimp}$, see also Theorem $7.24$ of \cite{Analy}). Since the filtration on $\overline{\mathcal{B}}$ is complete, we deduce that $\overline{\varphi}$ is an isomorphism (cf. Theorem $4$ ($5$), p.$31$ of \cite{Zaras}). We have $\overline{\mathcal{B}}$ complete because $\overline{\mathcal{B}}=\overline{\mathcal{A}}/\overline{\mathcal{R}}$, where $\overline{\mathcal{R}}$ is closed and therefore complete for the filtration induced from that of $\overline{\mathcal{A}}$.
\end{proof}
\begin{theorem}
\label{eq:maintheorem}
For $p>2$, the Iwasawa algebra $\Lambda(G(1))$is naturally isomorphic to $\mathcal{A}/\mathcal{R}$. 
\end{theorem}
\begin{proof}
  The reduction of $\varphi$ is $\overline{\varphi}$. We recall that $\overline{\mathcal{R}}$ is the image of $\mathcal{R}$ in $\overline{\mathcal{A}}$. Let $f \in \mathcal{A}$ satisfies $\varphi(f)=0$. We then have $\overline{f} \in \overline{\mathcal{R}}$ since $\overline{\mathcal{A}}/\overline{\mathcal{R}}\cong \Omega_{G(1)}$. So $f=r_1+pf_1, r_1 \in \mathcal{R},f_1 \in \mathcal{A}.$ Then $\varphi(f_1)=0$. Inductively, we obtain an expression $f=r_n+p^nf_n$ of the same type. Since $p^nf_n\rightarrow 0$ in $\mathcal{A}$ and $\mathcal{R}$ is closed, we deduce that $f \in \mathcal{R}$. 
\end{proof}
In conclusion for $p>2$ and $G$ simply connected, we have found a Lazard basis of $G(1)$ with respect to its $p$-valuation $\omega$ (see theorem $\ref{eq:orderedbasisthm}$). Furthermore, we have used it to construct explicit generators and relations for the Iwasawa algebra $\Omega_{(G(1))}$ with coefficients in $\mathbb{F}_p$ (see theorem $\ref{eq:secondimptheorem}$) and then we have lifted the coefficients to $\mathbb{Z}_p$ to construct the generators and relations of $\Lambda(G(1))$ (see theorem $\ref{eq:maintheorem}$). We might use the explicit presentation of the Iwasawa algebra, although we haven't done it already, to study the center of the Iwasawa algebra of $G(1)$ as done by Clozel in \cite{Laurent} for $SL_2(\mathbb{Z}_p)$, which will then provide a different proof of the results of Ardakov \cite{Ardakov}.

\bibliographystyle{amsplain}

\end{document}